\documentclass{elsarticle}
\usepackage[utf8]{inputenc}
\usepackage[hidelinks]{hyperref}
\usepackage{graphicx}
\usepackage{enumitem}

\usepackage{a4wide}
\usepackage{amsmath}
\usepackage{amsthm}
\usepackage{amsfonts}
\usepackage{amssymb}

\usepackage{subcaption}

\usepackage{enumitem}
\setlist[1]{itemsep=0em} 	

\renewcommand{\restriction}{\mathord{\upharpoonright}}

\theoremstyle{plain}
\newtheorem{theorem}{Theorem}[section]

\newtheorem{prop}[theorem]{Proposition}
\newtheorem{observation}[theorem]{Observation}
\newtheorem{lemma}[theorem]{Lemma}

\newtheorem{fact}[theorem]{Fact}
\newtheorem{claim}[theorem]{Claim}

\theoremstyle{definition}
\newtheorem{definition}[theorem]{Definition}
\newtheorem{example}{Example}
\newtheorem*{remark*}{Remark}
\newtheorem*{claim*}{Claim}
\theoremstyle{remark}
\newtheorem{remark}[theorem]{Remark}

\newcommand{\overbar}[1]{\mkern 1.5mu\overline{\mkern-1.5mu#1\mkern-1.5mu}\mkern 1.5mu}

\def\str#1{\mathbf {#1}}
\def\Fraisse{Fra\"{\i}ss\' e}

\def\dom{\mathop{\mathrm{Dom}}\nolimits}

\def\Aut{\mathop{\mathrm{Aut}}\nolimits}
\def\cl{\mathop{\mathrm{Cl}}\nolimits}

\def\str#1{\mathbf {#1}}
\def\arity#1{a(\rel{}{#1})}

\def\nbrel#1#2{R_{#1}^{#2}}
\def\rel#1#2{R_{\mathbf{#1}}^{#2}}
\def\permrel#1#2#3{#1(R^{#3})_{\mathbf{#2}}}

\def\func#1#2{F_{\mathbf{#1}}^{#2}}
\def\permfunc#1#2#3{#1(F^{#3})_{\mathbf{#2}}}

\def\GammaL{\Gamma\!_L}
\def\Aclass{\mathcal A^\delta_{K}}

\begin{document}
\bibliographystyle{plain}

\begin{frontmatter}
\title{Extending partial isometries of antipodal graphs}

\author{Mat\v ej Kone\v cn\'y}
\ead{matej@kam.mff.cuni.cz}
\address{Department of Applied Mathematics\\ Faculty of Mathematics and Physics, Charles University\\Malostransk\'e n\'am. 25, Prague, 118 00, Czech Republic}

\begin{abstract}
We prove EPPA (extension property for partial automorphisms) for all antipodal classes from Cherlin's list of metrically homogeneous graphs, thereby answering a question of Aranda et al. This paper should be seen as the first application of a new general method for proving EPPA which can bypass the lack of automorphism-preserving completions. It is done by combining the recent strengthening of the Herwig--Lascar theorem by Hubi\v cka, Ne\v set\v ril and the author with the ideas of the proof of EPPA for two-graphs by Evans et al.
\end{abstract}


\end{frontmatter}

\section{Introduction}
Let $G=(V,E)$ be a (not necessarily finite) graph and let $X,Y$ be subsets of $V$. We say that a function $f\colon X\to Y$ is a \emph{partial automorphism} of $G$ if $f$ is an isomorphism of $G[X]$ and $G[Y]$, the graphs induced by $G$ on $X$ and $Y$ respectively. This notion naturally extends to arbitrary structures (see Section~\ref{sec:preliminaries}).

In 1992 Hrushovski~\cite{hrushovski1992} proved that for every finite graph $G$ there is a finite graph $H$ such that $G$ is an induced subgraph of $H$ and every partial automorphism of $G$ extends to an automorphism of $H$.  This property is, in general, called the extension property for partial automorphisms (EPPA):

\begin{definition}
Let $\mathcal C$ be a class of finite structures. We say that $\mathcal C$ has the \emph{extension property for partial automorphisms} (or \emph{EPPA}), also called the \emph{Hrushovski property}, if for every $\str A\in \mathcal C$ there is $\str B\in \mathcal C$ such that $\str A$ is an (induced) substructure of $\str B$ and for every isomorphism $f$ of substructures of $\str A$ there is an automorphism $g$ of $\str B$ such that $f\subseteq g$. We call such $\str B$ an EPPA-witness for $\str A$.
\end{definition}

Hrushovski's proof was group-theoretical, Herwig and Lascar~\cite{herwig2000} later gave a simple combinatorial proof by embedding $G$ into the complement of a Kneser graph. After this, the quest of identifying new classes of structures with EPPA continued with a series of papers including~\cite{Aranda2017,Conant2015,eppatwographs,Herwig1995,herwig1998,herwig2000,hodkinson2003,Hubicka2018metricEPPA,Hubicka2017sauer,Hubicka2018EPPA,Konecny2018b,otto2017,solecki2005,vershik2008}.

Let $G=(V,E)$ be a graph. We say that a (partial) map $f\colon V\to V$ is \emph{distance-preserving} if whenever $u,v$ are in the domain of $f$, the distance between $u$ and $v$ is the same as the distance between $f(u)$ and $f(v)$. Clearly, every automorphism is distance-preserving. In 2005, Solecki~\cite{solecki2005} (and independently also Vershik~\cite{vershik2008}) proved that the class of all finite graphs has a variant of EPPA for distance-preserving maps. Namely, they proved that for every finite graph $G$ there is a finite graph $H$ satisfying the following:
\begin{enumerate}
\item $G$ is an induced subgraph of $H$,
\item whenever $u,v$ are vertices of $G$, then the distance between $u$ and $v$ in $G$ is the same as in $H$, and
\item every partial distance-preserving map of $G$ extends to an automorphism of $H$.
\end{enumerate}

It is not very convenient to work with distance-preserving maps, because they are relative to a graph and thus a distance-preserving map on a subgraph need not be distance-preserving with respect to a supergraph and vice versa. Given a graph $G=(V,E)$, it is more natural to consider the metric space $M=(V,d)$ where $d(u,v)$ is the number of edges of the shortest path from $u$ to $v$ in $G$ (we will call this the \emph{path-metric space} of $G$). And this is in fact what Solecki and Vershik did --- they proved EPPA for all (integer-valued) metric spaces, which is equivalent to EPPA for graphs with distance-preserving maps.

Vershik's proof is unpublished, Solecki's proof uses a complicated general theorem of Herwig and Lascar~\cite[Theorem~3.2]{herwig2000} about EPPA for structures with forbidden homomorphisms. Hubi\v cka, Ne\v set\v ril and the author~\cite{Hubicka2018metricEPPA} recently gave a simple self-contained proof of Solecki's result. There is also a group theoretical proof by Sabok~\cite{sabok2017automatic} using a construction \`a la Mackey~\cite{mackey1966}.

This paper continues in this direction. Generalising the concept of distance transitivity, we say that a (countable) connected graph $G$ is \emph{metrically homogeneous} if every partial distance-preserving map of $G$ with finite domain extends to an automorphism of $G$ (so it is, in a sense, an EPPA-witness for itself). Cherlin~\cite{Cherlin2011b} gave a list of countable metrically homogeneous graphs (which is conjectured to be complete and is provably complete in some cases~\cite{Amato2016, Cherlin2013}) in terms of classes of finite metric spaces which embed into the path-metric space of the given metrically homogeneous graph. EPPA and other combinatorial properties of classes from Cherlin's list were studied by Aranda, Bradley-Williams, Hng, Hubi{\v c}ka, Karamanlis, Kompatscher, Pawliuk and the author~\cite{Aranda2017a, Aranda2017c, Aranda2017} (see also~\cite{Konecny2018bc}) and in~\cite{Aranda2017} almost all the questions were settled, only EPPA for antipodal classes of odd diameter and bipartite antipodal classes of even diameter (see Section~\ref{sec:methom}) remained open. An important step was later done by Evans, Hubi\v cka, Ne\v set\v ril and the author~\cite{eppatwographs} who proved EPPA for antipodal metric spaces of diameter 3.

In this paper we combine the results of~\cite{Aranda2017} with the ideas from~\cite{eppatwographs} and the new strengthening of the Herwig--Lascar theorem by Hubi\v cka, Ne\v set\v ril and the author~\cite{Hubicka2018EPPA} (stated here in a weaker form as Theorem~\ref{thm:hl}) and prove the following theorem, thereby answering a question of Aranda et al. (Problem~1.3 in~\cite{Aranda2017}) and completing the study of EPPA for classes from Cherlin's list.
\begin{theorem}
\label{thm:main}
Every class of antipodal metric spaces from Cherlin's list has EPPA.
\end{theorem}

\section{Preliminaries}\label{sec:preliminaries}
A (not necessarily finite) structure $\str A$ is \emph{homogeneous} if every partial automorphism of $\str A$ with finite domain extends to a full automorphism of $\str A$ itself (so it is, in a sense, an EPPA-witness for itself). Gardiner proved~\cite{Gardiner1976} that the finite homogeneous graphs are precisely disjoint unions of cliques of the same size, their complements, the 5-cycle and the line graph of$K_{3,3}$. Lachlan and Woodrow later~\cite{Lachlan1980} classified the countably infinite homogeneous graphs. These are disjoint unions of cliques of the same size (possibly infinite), their complements, the Rado graph, the $K_n$-free variants of the Rado graph and their complements.

Every homogeneous structure can be associated with the class of all (isomorphism types of) its finite substructures, which is called its \emph{age}. By the \Fraisse{} theorem~\cite{Fraisse1953}, one can reconstruct the homogeneous structure back from this class (because it has the so-called \emph{amalgamation property}). For more on homogeneous structures see the survey by Macpherson~\cite{Macpherson2011}.

\medskip

A graph $G$ is \emph{vertex transitive} if for every pair of vertices $u,v$ there is an automorphism sending $u$ to $v$, it is \emph{edge transitive} if every edge can be sent to every other edge by an automorphism and it is \emph{distance transitive} if for every two pairs of vertices $u,v$ and $x,y$ such that the distance between $u$ and $v$ is the same as the distance between $x$ and $y$ there is an automorphism sending $u$ to $x$ and $v$ to $y$.

Distance transitivity is a very strong condition. For example, there are only finitely many finite 3-regular distance transitive graphs~\cite{Biggs1971} and the full catalogue is available in some other particular cases. However, for larger degrees, the classification is unknown, see e.g. the book by Godsil and Royle~\cite{Godsil2001}, largely devoted to the study of distance transitive graphs.

Recall that a connected graph $G$ is metrically homogeneous if every partial distance-preserving map of $G$ with finite domain extends to an automorphism of $G$. This is equivalent to saying that the path-metric space of $G$ is homogeneous in the sense of the previous paragraphs. All connected homogeneous graphs are also metrically homogeneous, because every pair of vertices is either connected by an edge or by a path of length 2. Finite cycles of size at least 6 are examples of metrically homogeneous graphs which are not homogeneous.



\begin{remark}
If one checks the known classes with EPPA, they will find out that they all are ages of homogeneous structures. This is not a coincidence. It is easy to see that if a class of finite structures $\mathcal C$ has EPPA and the \emph{joint embedding property} (for every $\str A, \str B\in \mathcal C$ there is $\str C\in \mathcal C$ which contains a copy of both of them), then $\mathcal C$ is the age of a homogeneous structure provided that it contains at most countably many members up to isomorphism. This restricts the candidate classes for EPPA severely and connects finite combinatorics with the study of infinite homogeneous structures and infinite permutation groups.

In the other direction, EPPA has some implications for the automorphism group (with the \emph{pointwise convergence topology}) of the corresponding homogeneous structure, see for example the paper of Hodges, Hodkinson, Lascar, and Shelah~\cite{hodges1993b}.
\end{remark}

\subsection{$\GammaL$-structures}
An important feature of the strengthening of the Herwig--Lascar theorem by Hubi\v cka, Ne\v set\v ril and the author (Theorem~\ref{thm:hl}) is that it allows to also permute the language. Namely, we will work with categories whose objects are the standard model-theoretic structures (in a given language), but the arrows are potentially richer, allowing a permutation of the language. The reader is invited to verify that in the following paragraphs, if the group $\GammaL$ consists of the identity, one obtains the usual notion of model-theoretic $L$-structures with the corresponding maps.

The following notions are taken from~\cite{Hubicka2018EPPA}, sometimes stated in a more special form which is sufficient for our purposes. Many of them were introduced by Hubi\v cka and Ne\v set\v ril~\cite{Hubicka2016} in the context of structural Ramsey theory (e.g. homomorphism-embeddings or completions).

Let $L=L_\mathcal R\cup L_\mathcal F$ be a language with relational symbols $\rel{}{}\in L_\mathcal R$, each having associated {\em arities} denoted by $\arity{}$ and function symbols $F\in L_\mathcal F$. All functions in this paper are unary and have unary range. Let $\GammaL$ be a permutation group on $L$ such that each $\alpha\in\GammaL$ preserves the partition $L=L_\mathcal R\cup L_\mathcal F$ (that is, maps relations to relations and functions to functions) and the arities of all symbols. We will say that $\GammaL$ is a \emph{language equipped with a permutation group}. 

A \emph{$\GammaL$-structure} $\str{A}$ is a structure with {\em vertex set} $A$, functions $\func{A}{}\colon A\to A$ for every $\func{}{}\in L_\mathcal F$ and relations $\rel{A}{}\subseteq A^{\arity{}}$ for every $\rel{}{}\in L_\mathcal R$. We will write structures in bold and their corresponding vertex sets in normal font. If $\GammaL$ is trivial, we will often talk about $L$-structures instead of $\GammaL$-structures.

If the set $A$ is finite we call $\str A$ a \emph{finite structure}. If the language $L$ contains no function symbols, we call $L$ a {\em relational language} and say that a $\GammaL$-structure is a {\em relational $\GammaL$-structure}.

A \emph{homomorphism} $f\colon \str{A}\to \str{B}$ is a pair $f=(f_L,f_A)$ where $f_L\in \GammaL$ and $f_A$ is a mapping $A\to B$
 such that  for every $\rel{}{}\in L_\mathcal R$ and $\func{}{}\in L_\mathcal F$ we have:
\begin{enumerate}
\item[(a)] $(x_1,x_2,\ldots, x_{\arity{}})\in \rel{A}{}\implies (f_A(x_1),f_A(x_2),\ldots,f_A(x_{\arity{}}))\in \permrel{f_L}{B}{}$, and
\item[(b)] $f_A(\func{A}{}(x))=\permfunc{f_L}{B}{}(f_A(x))$.
\end{enumerate}
For brevity, we will also write $f(x)$ for $f_A(x)$ in the context where $x\in A$ and $f(S)$ for $f_L(S)$ where $S\in L$.
For a subset $A'\subseteq A$ we denote by $f(A')$ the set $\{f(x): x\in A'\}$ and by $f(\str{A})$ the homomorphic image of a structure $\str{A}$. Note that we write $f\colon\str A\to \str B$ to emphasize that $f$ respects the structure.

If $f_A$ is injective then $f$ is called a \emph{monomorphism}. A monomorphism $f$ is an \emph{embedding} if for every $\rel{}{}\in L_\mathcal R$ we have the equivalence in the definition, that is, 
$$(x_1,x_2,\ldots, x_{\arity{}})\in \rel{A}{}\iff (f(x_1),f(x_2),\ldots,f(x_{\arity{}}))\in \permrel{f}{B}{}.$$ If the inclusion $A\subseteq B$ together with the identity of $\GammaL$ form an embedding, we say that $\str{A}$ is a \emph{substructure} of $\str{B}$ and often denote it as $\str A\subseteq \str B$. For an embedding $f\colon\str{A}\to \str{B}$ we say that $f(\str{A})$ is a \emph{copy} of $\str A$ in $\str B$. If $f$ is an embedding where $f_A$ is onto, then $f$ is an \emph{isomorphism} and an isomorphism $\str A\to\str A$ is called an \emph{automorphism}.

Note that from the previous paragraph it follows that when $L$ contains functions, not every subset of vertices induces a substructure. Namely, every substructure needs to be closed on functions. For example, if $L$ consists of one unary function $F$, $\GammaL$ contains only the identity and $\str B$ is a $\GammaL$-structure with vertex set $B=\{b_1,b_2\}$ such that $F(b_1)=b_2$ and $F(b_2)$ is not defined, then there is a substructure of $\str B$ on the set $\{b_2\}$, but the smallest substructure of $\str B$ containing $b_1$ is $\str B$ itself. Generalising this example, we say that for a $\GammaL$-structure $\str{B}$ and a set $A$ which is a subset of $B$, the {\em closure of $A$ in $\str{B}$}, denoted by $\cl_\str{B}(A)$, is the smallest substructure of $\str{B}$ containing $A$. For $x\in B$, we will also write $\cl_\str{B}(x)$ for $\cl_\str{B}(\{x\})$.

Generalising the notion of a graph clique, we say that a $\GammaL$-structure $\str A$ is
\emph{irreducible} if for every pair of distinct vertices $x,y\in A$ there is a relation $\rel{}{}\in L$ and a tuple $\bar{r}\in A^{\arity{}}$ containing both $x$ and $y$ such that $\bar{r} \in \rel{A}{}$. Note that the definition of irreducibility from~\cite{Hubicka2018EPPA} is more general than this one (making more structures irreducible in general languages with functions), but stating it would need some more preliminary definitions and moreover they are equivalent for structures which we will consider in this paper.

\begin{example}
If the language only contains unary relations, irreflexive symmetric binary relations and unary functions (which will always be true in this paper), a structure is irreducible if and only if the union of the binary relations is a complete graph.
\end{example}

A homomorphism $f\colon\str{A}\to \str{B}$ is
a \emph{homomorphism-embedding} if the restriction $f\restriction_{\str C}$ is an embedding whenever $\str C$ is an irreducible
substructure of $\str{A}$.

\subsection{EPPA for $\GammaL$-structures}
We next state the main result of~\cite{Hubicka2018EPPA} for which we need the following definitions, which are mostly variants of the definitions needed for the Hubi\v cka--Ne\v set\v ril theorem~\cite{Hubicka2016}.

A {\em partial automorphism} of a $\GammaL$-structure $\str{A}$ is an isomorphism $f\colon
\str{C} \to \str{C}'$ where $\str{C}$ and $\str{C}'$ are substructures of
$\str{A}$ (remember that it also includes a permutation of the language which is not partial).  We say that a class $\mathcal C$ of finite $\GammaL$-structures  has the {\em
extension property for partial automorphisms}  ({\em EPPA}) if for every $\str{A} \in \mathcal C$
there is $\str{B} \in \mathcal C$ such that $\str{A}$ is a substructure of $\str{B}$
and every partial automorphism of $\str{A}$ extends to an automorphism of
$\str{B}$.  We call $\str{B}$ with such a property an {\em EPPA-witness for
$\str{A}$}.
If $\str{B}$ is an EPPA-witness for $\str A$, we say that it is \emph{irreducible-structure faithful} if for every irreducible substructure $\str{C}$ of $\str{B}$ there exists an 
automorphism $g$ of $\str{B}$ such that $g(C)\subseteq A$. We say that a class $\mathcal C$ of finite $\GammaL$-structures \emph{has EPPA} if there is an EPPA-witness $\str B\in\mathcal C$  for every $\str A\in \mathcal C$. We say that $\mathcal C$ \emph{has irreducible-structure faithful EPPA} if the witness can always be chosen to be irreducible-structure faithful.

\begin{example}\leavevmode
\begin{enumerate}
\item Let $\str A$ be the graph on vertices $u,v,w$ containing a single edge $uv$ (here, the language consists of one binary relation and the permutation group is trivial). Then a possible (irreducible-structure faithful) EPPA-witness for $\str A$ is the graph $\str B$ on vertices $u,v,w,x$ with edges $uv$ and $wx$.
\item To see an example of a non-trivial permutation group, let $L$ be the language consisting of unary relations $\rel{}{i}$, where $1\leq i\leq 10$ and let $\GammaL$ consist of all permutation of $L$ which fix $\rel{}{10}$. Let $\str A$ be the $\GammaL$ structure on one vertex $v$ such that $\rel{A}{1} = \{v\}$ and $\rel{A}{i}=\emptyset$ for every $i\geq 2$. Then every EPPA-witness $\str B$ for $\str A$ must contain vertices $v_2,\ldots,v_9$ such that $v_i\in\rel{B}{i}$ for every $2\leq i\leq 9$, because $\str B$ needs to extend all partial automorphism $f^i$, $2\leq i\leq 9$, such that $f^i_A$ is the empty function and $f^i_L\in\GammaL$ sends $\rel{}{1}$ to $\rel{}{i}$.
\end{enumerate}
\end{example}

\begin{definition}
\label{defn:completion}
Let $\str{C}$ be a $\GammaL$-structure. A $\GammaL$-structure $\str{C}'$ is a \emph{completion}
of $\str{C}$ if there exists an injective homomorphism-embedding $f\colon \str{C}\to\str{C}'$ which fixes every symbol of the language. We say that $\str C'$ is an \emph{automorphism-preserving completion} of $\str C$, if $C\subseteq C'$, the inclusion together with the identity from $\GammaL$ give a homomorphism-embedding, for every $\alpha\in \Aut(\str C)$ there is $\beta\in \Aut(\str C')$ such that $\alpha\subseteq \beta$ and moreover the map $\alpha\mapsto\beta$ is a group homomorphism $\Aut(\str{C})\to\Aut(\str{C}')$.
\end{definition}
In this paper, the languages will contain only unary and binary relations and unary functions, and moreover, whenever $\str C'$ will be a completion of $\str C$, it will always hold that $C' = C$ and that the identity is a homomorphism-embedding. In such a case, for every relation $\rel{}{}\in L$ we have $\rel{C}{} \subseteq \nbrel{\str C'}{}$ with equality for unary $\rel{}{}$. For binary $\rel{}{}$ it holds that if $(u,v)\in \nbrel{\str C'}{}\setminus\rel{C}{}$, then for every binary $\rel{}{0}\in L$ we have $(u,v)\notin \rel{C}{0}$. Furthermore, $\str C'$ is an automorphism-preserving completion of $\str C$ if and only if $\Aut(\str C')=\Aut(\str C)$.

\begin{figure}[t]
\begin{subfigure}{.3\textwidth}
  \centering
  \includegraphics[width=.8\linewidth]{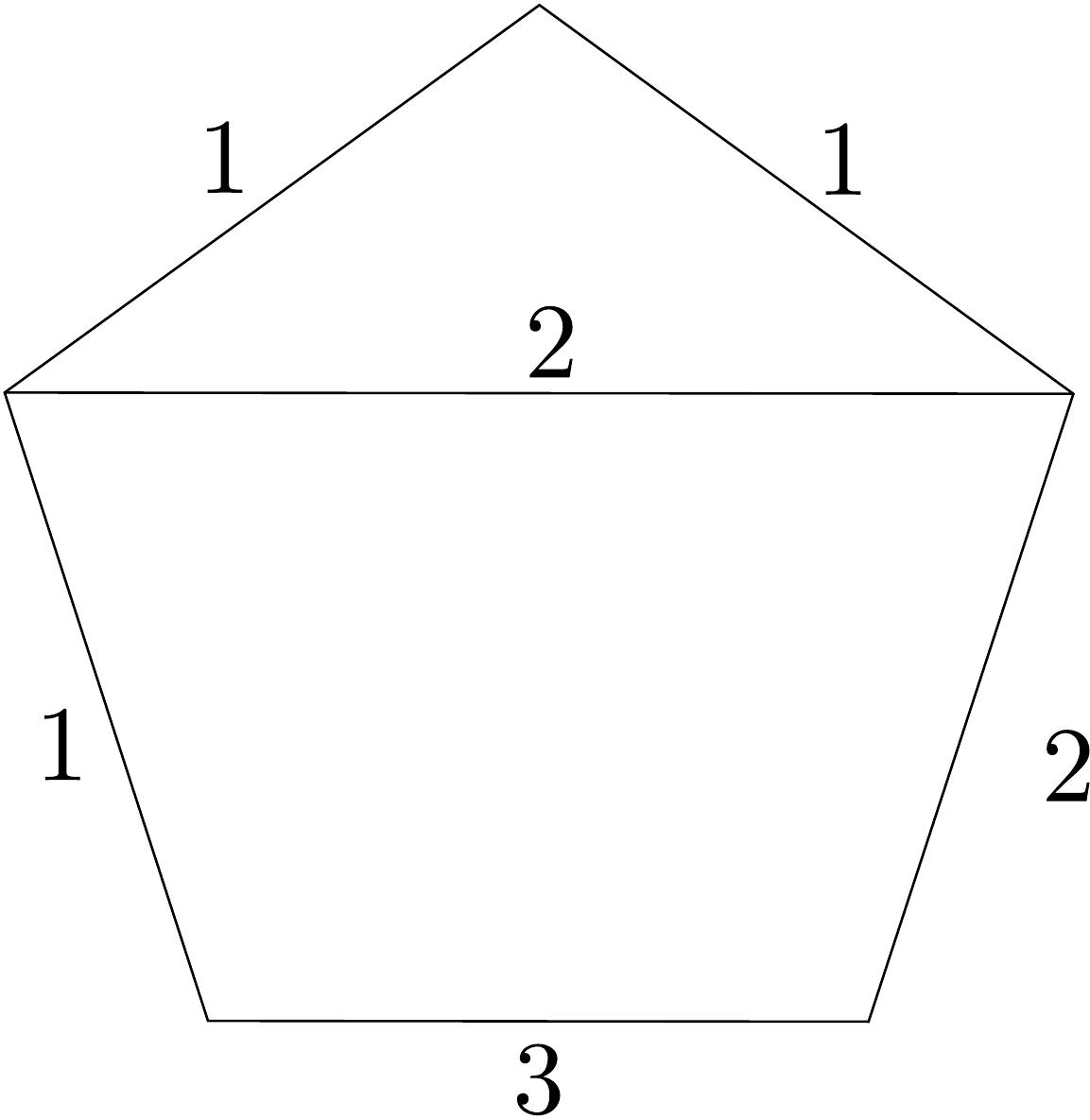}  
  \caption{}
  \label{fig:completion:1}
\end{subfigure}
\begin{subfigure}{.3\textwidth}
  \centering
  \includegraphics[width=.8\linewidth]{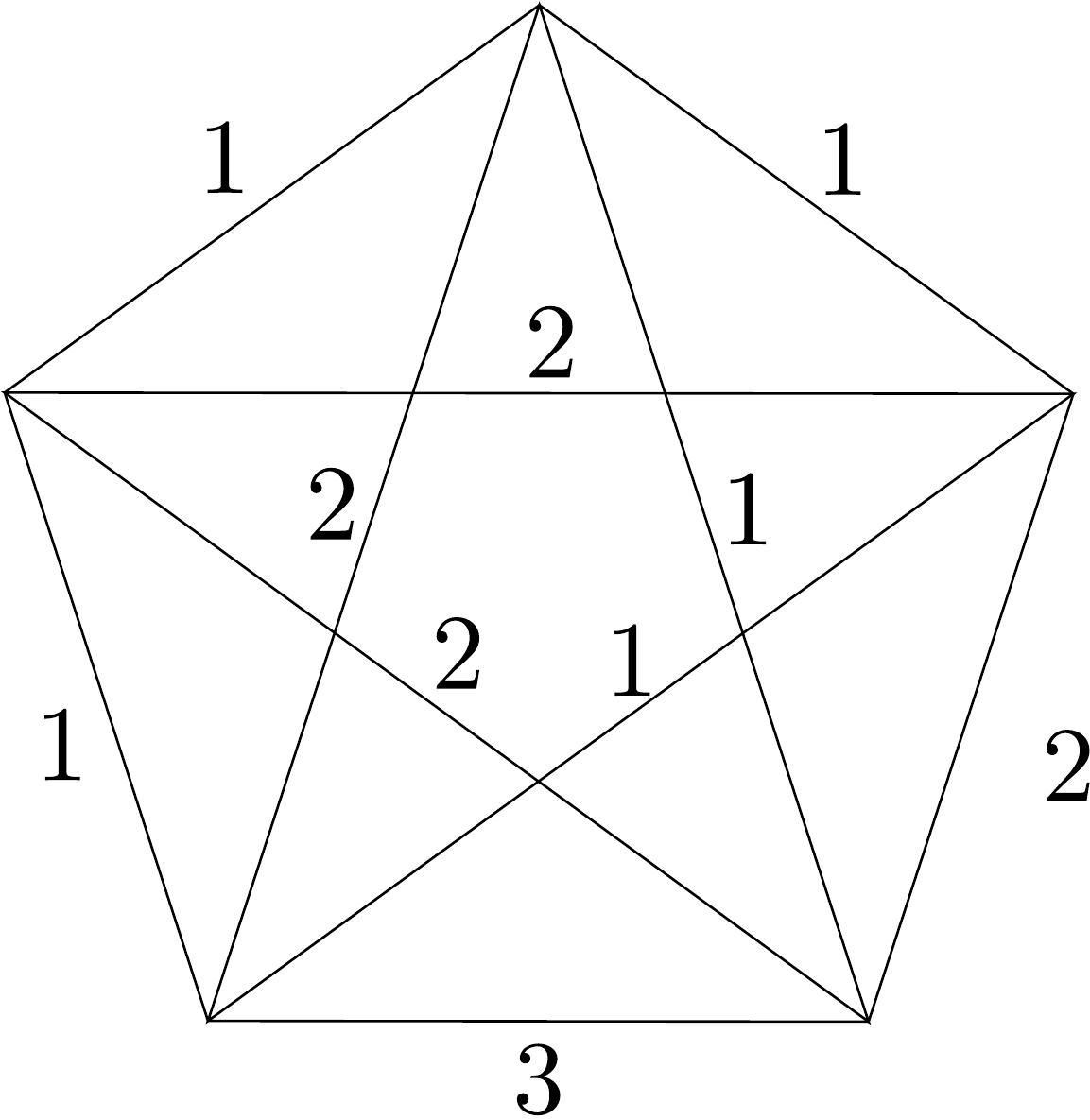}  
  \caption{}
  \label{fig:completion:2}
\end{subfigure}
\begin{subfigure}{.3\textwidth}
  \centering
  \includegraphics[width=.8\linewidth]{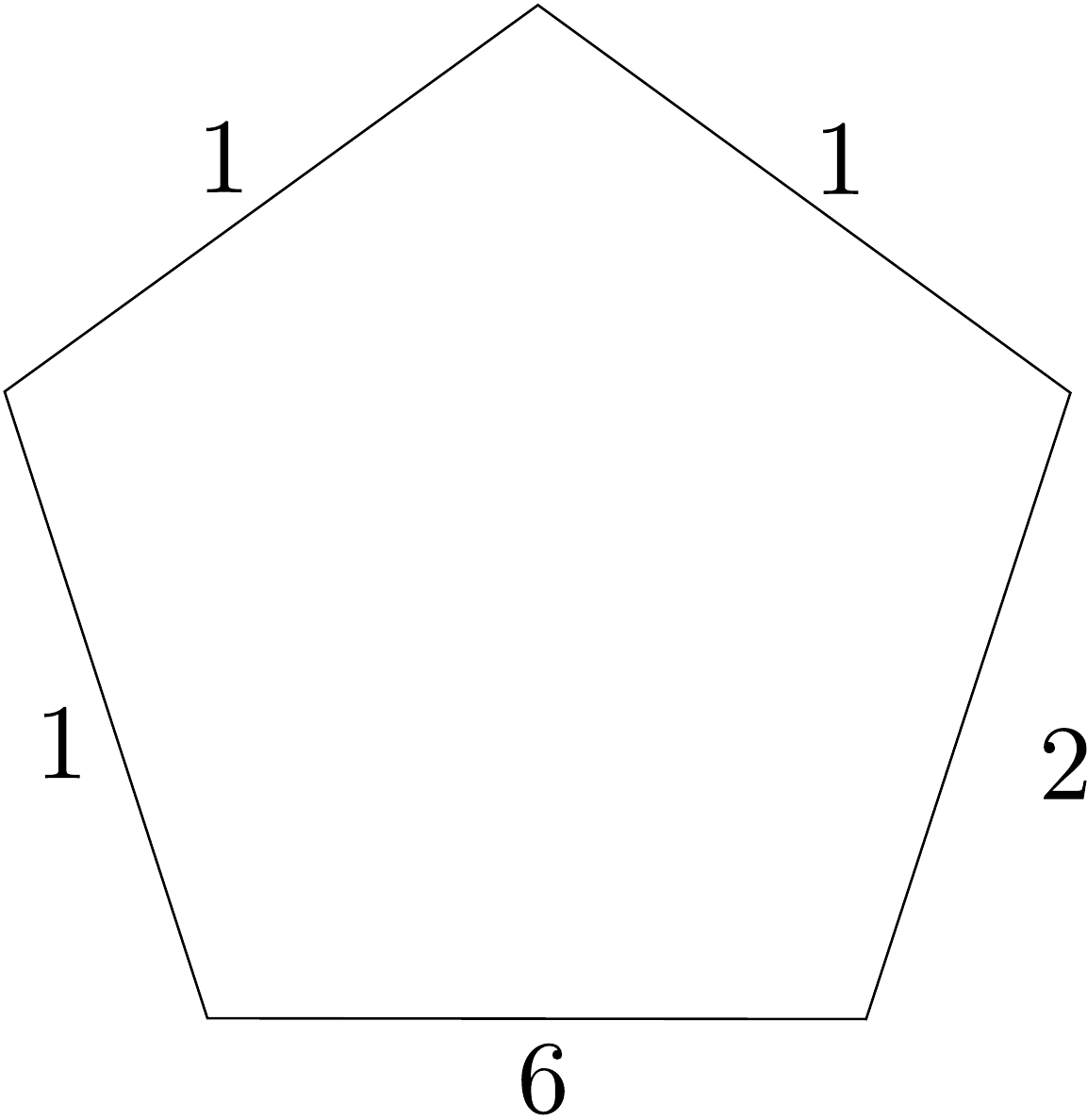}  
  \caption{}
  \label{fig:completion:3}
\end{subfigure}
\caption{Completions}
\label{fig:completion}
\end{figure}

\begin{example}\label{ex:matricspaces}
Consider the class $\mathcal C_\mathbb N$ of all finite integer-valued metric spaces understood as structures in a binary symmetric relational language $L$ with a relation for every nonzero distance (the fact that $d(x,x)=0$ is implicit). In Figure~\ref{fig:completion} we see the following:
\begin{enumerate}
\item[(\subref{fig:completion:1})] An $L$-structure which has an automorphism-preserving completion in $\mathcal C$,
\item[(\subref{fig:completion:2})] one such completion, and
\item[(\subref{fig:completion:3})] an $L$-structure which has no completion in $\mathcal C$.
\end{enumerate}
\end{example}

\begin{definition}\label{defn:locallyfinite}
Let $L$ be a finite language with relations and unary functions equipped with a permutation group $\GammaL$.
Let $\mathcal E$ be a class of finite $\GammaL$-structures and let $\mathcal K$ be a subclass of $\mathcal E$ consisting of irreducible structures. We say
that $\mathcal K$ is a \emph{locally finite subclass of $\mathcal E$} if for every $\str A\in \mathcal K$ and every $\str{B}_0 \in \mathcal E$ there is a finite integer $n = n(\str A, \str {B}_0)$ such that 
every $\GammaL$-structure $\str B$ has a completion $\str B'\in \mathcal K$ provided that it satisfies the following:
\begin{enumerate}
\item\label{lfcond1} For every vertex $v\in B$ we have that $\cl_{\str B}(v)$ lies in a copy of $\str A$,
\item there is a homomorphism-embedding from $\str{B}$ to $\str{B}_0$, and
\item every substructure of $\str{B}$ with at most $n$ vertices has a comple\-tion in $\mathcal K$.
\end{enumerate}
We say that $\mathcal K$ is a \emph{locally finite automorphism-preserving subclass of $\mathcal E$} if
in the condition above, the completion of $\str B$ can always be chosen to be automorphism-preserving.
\end{definition}
\begin{remark}
While in Definition~\ref{defn:locallyfinite} we promise that $\cl_{\str B}(v)$ lies in a copy of $\str A$, the definition of local finiteness for the Hubi\v cka--Ne\v set\v ril theorem (Definition~2.4 from~\cite{Hubicka2016}, similarly also the definition of local finiteness from~\cite{Hubicka2018EPPA}) promises that every irreducible substructure of $\str B$ comes from $\mathcal K$. The difference here is due to the fact that the definition of irreducibility is simplified in this paper and does not work well for general languages with functions. In the applications, both conditions are used to ensure that closures behave well in $\str B$.
\end{remark}

\begin{example}\label{ex:locallyfinite}
Let us observe that the class $\mathcal C_\mathbb N$ from Example~\ref{ex:matricspaces} is a locally finite automorphism-preserving subclass of the class $\mathcal E$ consisting of all finite $L$-structures (for $L$ from Example~\ref{ex:matricspaces}), where all relations are symmetric and irreflexive and every pair of vertices is in at most one relation.
Fix $\str A\in \mathcal K$ and $\str B_0\in \mathcal E$. The assumption on $\mathcal E$ justifies defining a symmetric partial function $d_{\str B_0}\colon B_0^2\to \mathbb N$ where $d(u,u)=0$ and $d(u,v)=\ell$ if and only if $u$ and $v$ are in the relation corresponding to $\ell$ in $\str B_0$. Let $S$ be the set of all integers $\ell$ for which there are vertices $u,v\in B_0$ such that $d(u,v)=\ell$. Since $\str B_0$ is finite, $S$ is also finite. Put $n=\max_{a,b\in S}\lceil\frac{a}{b}\rceil$.

Let $\str B$ be an $L$-structure satisfying the conditions of Definition~\ref{defn:locallyfinite} (since $L$ contains no functions, condition~\ref{lfcond1} is satisfied trivially). The existence of a homomorphism-embedding from $\str B$ to $\str B_0$ implies that all the relations in $\str B$ are also symmetric and irreflexive and every pair of vertices of $\str B$ is in at most one relation, hence we can analogously define a partial function $d_\str B\colon B^2\to \mathbb N$. Moreover, since there is a homomorphism-embedding from $\str B$ to $\str B_0$, we also get that the only non-empty distance relations in $\str B$ are those representing distances from $S$.

Next we define function $d'\colon B^2\to \mathbb N$ by
$$d'(x,y) = \min\limits_{\str P\text{ is a path $x\to y$ in $\str B$}} \|\str P\|,$$
where by a path $x\to y$ we mean a sequence of distinct vertices $x=p_1, \ldots, p_k=y$ such that $d_\str{B}(p_i,p_{i+1})$ is defined for every $i$ satisfying $1\leq i<k$, and we define $\|\str P\|$ as $\sum_{i=1}^{k-1}d_\str{B}(p_i,p_{i+1})$. It is easy to verify that $(B,d')$ is a metric space with distances from $\mathbb N$ and that $d_\str{B} \subseteq d'$ if and only if $\str B$ contains no \emph{non-metric cycles}, that is, sequences of vertices $v_1,\ldots, v_k$ such that $d_\str{B}(v_i,v_{i+1})$ is defined for every $1\leq i \leq k$ (we identify $v_{k+1}=v_1$) and $d_\str B(v_1,v_k) > \sum_{i=1}^{k-1} d_\str{B}(v_i, v_{i+1})$.

Since, clearly, non-metric cycles do not have a completion in $\mathcal C_\mathbb N$, it follows from the definition of $n$ that $\str B$ contains no non-metric cycles and hence has a completion $\str B'=(B,d')$ in $\mathcal C_\mathbb N$ as requested. Moreover, from the canonicity of the definition of $d'$ it follows that this completion is automorphism-preserving and hence we have proved that $\mathcal C_\mathbb N$ is an automorphism-preserving completion of $\str B$.

This construction of $\str B'$ is called the \emph{shortest-path completion} in~\cite{Hubicka2016} and was already used by Solecki~\cite{solecki2005} to prove EPPA for the class of all finite metric spaces and by Ne\v set\v ril~\cite{Nevsetvril2007} to find a Ramsey expansion of the class of all finite metric spaces.
\end{example}

The main theorem of~\cite{Hubicka2018EPPA} can be stated as follows.
\begin{theorem}[\cite{Hubicka2018EPPA}]
\label{thm:hl}
Let $L$ be a finite language with relations and unary functions equipped with a permutation group $\GammaL$, let $\mathcal E$ be a class of finite $\GammaL$-structures which has irreducible-structure faithful EPPA and let $\mathcal K$ be a hereditary
locally finite automorphism-preserving subclass of $\mathcal E$ with the strong amalgamation property, which consists of irreducible structures.
Then $\mathcal K$ has EPPA.
\end{theorem}
Here $\mathcal K$ is \emph{hereditary} if whenever $\str B\in \mathcal K$ and $\str A\subseteq \str B$, then also $\str A\in\mathcal K$. We will not define what the \emph{strong amalgamation property} is (see~\cite{Hubicka2018EPPA}), but all classes for which we will use Theorem~\ref{thm:hl} will have this property.

Since Theorem~\ref{thm:hl} has a form of implication, we will need the following theorem from~\cite{Hubicka2018EPPA} to supply us with the base EPPA class $\mathcal E$.
\begin{theorem}[\cite{Hubicka2018EPPA}]
\label{thm:eppanoaxioms}
Let $L$ be a finite language with relations and unary functions equipped with a permutation group $\GammaL$. Then the class of all finite $\GammaL$-structures has irreducible-structure faithful EPPA.
\end{theorem}

Note that combining Example~\ref{ex:locallyfinite} with Theorems~\ref{thm:hl} and~\ref{thm:eppanoaxioms} gives a proof of Solecki's result that finite metric spaces have EPPA (to prove that $\mathcal E$ has EPPA, one has to use Theorem~\ref{thm:eppanoaxioms} for a finite fragment of $L$ to get an EPPA-witness for a given $\str A\in \mathcal E$). Both Theorems~\ref{thm:hl} and~\ref{thm:eppanoaxioms} are proved by an application of the method of valuation functions, a variant of which we will also use in this paper. More precisely, Theorem~\ref{thm:eppanoaxioms} is proved by giving an explicit construction of EPPA-witnesses. Theorem~\ref{thm:hl} iteratively applies the method of valuation functions to produce, given $\str A\in\mathcal K$ and its irreducible-structure faithful EPPA-witness $\str B_0\in \mathcal E$, an EPPA-witness $\str B$ satisfying the conditions of Definition~\ref{defn:locallyfinite}, the automorphism-preserving completion $\str B'$ of $\str B$ is then the desired EPPA-witness for $\str A$ in $\mathcal K$.

\subsection{Metrically homogeneous graphs}\label{sec:methom}
Most of the details of Cherlin's metric spaces are not important for this paper. We only give the necessary definitions and facts and refer the reader to~\cite{Cherlin2011b}, \cite{Aranda2017} or~\cite{Konecny2018bc}.

All the metric spaces we will work with have distances from $\{0,1,\ldots,\delta\}$ for some integer $\delta$. Therefore, we will view them interchangeably as pairs $(A,d)$ where $d$ is the metric, as complete graphs with edges labelled by $\{1,\ldots,\delta\}$ (we will call these \emph{complete $[\delta]$-edge-labelled graphs}) where the labels of every triangle satisfy the triangle inequality, and as relational structures with trivial $\GammaL$ and binary symmetric irreflexive relations $R^1,\ldots, R^\delta$ (distance $0$ is not represented) such that every pair of vertices is in exactly one relation and the triangle inequality is satisfied. The middle point of view works best with the notion of completions: Given a (not necessary complete) $[\delta]$-edge-labelled graph $\str G$, a $[\delta]$-edge-labelled graph $\str G'$ is a completion of $\str G$ if $\str G$ is a non-induced subgraph of $\str G'$ and the labels are preserved.

We will say that two vertices are at distance $a$ and that they are connected by an edge of length $a$ interchangeably. In particular, when we talk about an edge of a $[\delta]$-edge-labelled graph, we mean a pair of vertices such that their distance is defined, it does not necessarily mean that they are at distance 1.

A major part of Cherlin's list of the classes of finite metric spaces which embed into the path-metric space of a countably infinite metrically homogeneous graph consists of certain 5-parameter classes $\mathcal A^\delta_{K_1,K_2,C_0,C_1}$. These are classes of metric spaces with distances $\{0,1,\ldots,\delta\}$ (we call $\delta$ the \emph{diameter} of such spaces) omitting certain families of triangles (e.g. triangles of short odd perimeter or triangles of long even perimeter).

A special case of these classes are the \emph{antipodal} classes, where the five parameters have only two degrees of freedom. Here we will denote the antipodal classes as $\mathcal A^\delta_K$, where $1\leq K\leq \frac{\delta}{2}$, or $K=\delta$.\footnote{If $K\neq \delta$, the other parameters are then defined as $K_1=K$, $K_2 = \delta-K$, $C_0=2\delta+2$ and $C_1=2\delta+1$, if $K=\delta$, then $K_1=\infty$ and the other parameters are as before.} $\mathcal A^\delta_K$ is defined as the class of all finite metric spaces with distances from $\{0,1,\ldots,\delta\}$ such that they contain no triangle with distances $a,b,c$ for which at least one of the following holds:
\begin{enumerate}
\item $a+b+c > 2\delta$,
\item $a+b+c$ is odd and $a+b+c < 2K$, or
\item $a+b+c$ is odd and $a+b+c > 2(\delta-K) + 2\min(a,b,c)$.
\end{enumerate}
However, for our purposes, we need only the following fact:
\begin{fact}[Antipodal spaces]\label{fact:antipodal}
The following holds in every class $\Aclass$ of antipodal metric spaces from Cherlin's list:
\begin{enumerate}
  \item The edges of length $\delta$ form a matching (that is, for every vertex there is at most one vertex at distance $\delta$ from it) and for every $\str A\in\Aclass$ there is a unique $\str B\in\Aclass$ such that $\str A\subseteq \str B$, the edges of length $\delta$ form a perfect matching in $\str B$ and every edge of length $\delta$ in $\str B$ has at least one endpoint from $\str A$.
  \item\label{pt2} For every pair of vertices $u,v$ such that $d(u,v)=\delta$ and for every vertex $w$ we have $d(u,w)+d(v,w)=\delta$.
  \item If one selects exactly one vertex from each edge of length $\delta$, the metric space they induce belongs to a special (non-antipodal) class of diameter $\delta-1$ which we will call $\mathcal B^\delta_K$.\footnote{It is in fact $\mathcal A^{\delta-1}_{K_1,K_2,C_0,C_1}$ for $K_1=K$, $K_2 = \delta-K$, $C_0=2\delta+2$ and $C_1=2\delta+1$.} And the other way around, one can get an antipodal metric space from every metric space $\str M\in \mathcal B^\delta_K$ by taking two disjoint copies of $\str M$, connecting every vertex to its copy by an edge of length $\delta$ and using point~\ref{pt2} to fill-in the missing distances.
\end{enumerate}
\end{fact}

There are two kinds of antipodal classes with different combinatorial behaviour --- those that come from a countable bipartite metrically homogeneous graph (they correspond to the case $K=\delta$) and those that come from a non-bipartite one. We will call the first the \emph{bipartite} classes (their members have the property that they contain no triangles, or more generally cycles, of odd perimeter) and we will call the others the \emph{non-bipartite} ones. This is slightly misleading, because some of the finite metric subspaces of the path-metric of a non-bipartite metrically homogeneous graph are surely bipartite, but it should not cause any confusion in this paper. The non-bipartite class of antipodal metric spaces of diameter $3$ is closely connected to switching classes of graphs and to two-graphs (see~\cite{eppatwographs}).

The following fact summarizes results from~\cite{Aranda2017} about the non-bipartite odd diameter antipodal classes.
\begin{fact}\label{fact:completion}
Let $\Aclass$ be a non-bipartite class of antipodal metric spaces of odd diameter $\delta$. Let $\str A$ be a $[\delta]$-edge-labelled graph such that the edges of length $\delta$ of $\str A$ form a perfect matching and furthermore for every $u,v,w\in A$ such that $d_\str{A}(u,v)=\delta$ and $w\neq u,v$, either $w$ is not connected by an edge to either of $u,v$, or $d_\str{A}(u,w)+d_\str{A}(v,w)=\delta$. Suppose furthermore that $\str A$ contains none of the finitely many cycles forbidden in $\mathcal B^\delta_K$.

Let $f\colon {A\choose 2}\to \{0,1\}$ be a mapping satisfying the following.
\begin{enumerate}
  \item Whenever $uv$ is an edge of $\str A$, then $f(uv) \equiv d_\str{A}(u,v)\mod 2$.
  \item Let $u_1v_1$ and $u_2v_2$ be two different edges of length $\delta$ of $\str A$. Then $f(u_1u_2)=f(v_1v_2)$, $f(u_1v_2)=f(u_2v_1)$ and $f(u_1u_2)\neq f(u_1v_2)$.
\end{enumerate}

Then there is $\overbar{\str A}\in \Aclass$ such that the following holds.
\begin{enumerate}
  \item $\overbar{\str A}$ is a completion of $\str A$ with the same vertex set,
  \item for every edge $uv$ of $\overbar{\str A}$ it holds that $f(uv) \equiv d_{\overbar{\str{A}}}(u,v)\mod 2$, and
  \item Every automorphism of $\str A$ which preserves values of $f$ is also an automorphism of $\overbar{\str A}$.
\end{enumerate}
\end{fact}
Such $\overbar{\str A}$ can be constructed by picking one vertex from each edge of length $\delta$, considering this auxiliary metric space of diameter $\delta-1$, completing it using Theorem~4.9 from~\cite{Aranda2017} (see also Lemma~4.18 from the same paper, or~\cite{Hubickacycles2018}) and then pulling this completion back using $f$ to decide parities of the edges. The proof then uses the observation that the completion procedure for $\mathcal A^{\delta-1}_{K_1,K_2,C_0,C_1}$ from~\cite{Aranda2017} preserves the equivalence ``$a\sim \delta-a$''. That is, we say that two $[\delta-1]$-edge-labelled graphs $\str G$ and $\str G'$ are equivalent if they share the same vertex set and the same edge set and every edge has either the same label in both $\str G$ and $\str G'$, or it has label $a$ in $\str G$ and $\delta-a$ in $\str G'$. The completion procedure then produces equivalent graphs whenever given equivalent graphs.

\section{The odd diameter non-bipartite case}\label{sec:nonbip}
EPPA for the even diameter non-bipartite case was proved in~\cite{Aranda2017}. In this section we prove the following proposition.
\begin{prop}\label{prop:oddnonbip}
Let $\Aclass$ be a non-bipartite class of antipodal metric spaces of odd diameter. Then for every $\str A\in \Aclass$ there is $\str B\in \Aclass$ which is an EPPA-witness for $\str A$.
\end{prop}
Proposition~\ref{prop:oddnonbip} extends the results of~\cite{eppatwographs} where it was proved for diameter 3.

\subsection{Motivation}\label{sec:motivation}
We first give some motivation and intuition behind Proposition~\ref{prop:oddnonbip}, as its proof is a bit technical. Consider the class $\mathcal A^3_1$. It consists of all finite complete $[3]$-edge-labelled graphs which omit triangles with distances $(1,1,3)$, $(2,2,3)$ and $(3,3,a)$, where $1\leq a\leq 3$. In other words, the edges of length 3 form a matching (and by Fact~\ref{fact:antipodal} we can assume that it is a perfect matching), and if $u,v,w,x$ are pairwise distinct vertices such that $d(u,v)=d(w,x)=3$, then they form an \emph{antipodal quadruple}, which means that $d(u,w)=d(v,x)$, $d(u,x)=d(v,w)$, and either $d(u,w)=1$ and $d(u,x)=2$, or $d(u,w)=2$ and $d(u,x)=1$ (see Figure~\ref{fig:quadruple}).

\begin{figure}[t]
\centering
\includegraphics[width=.8\linewidth]{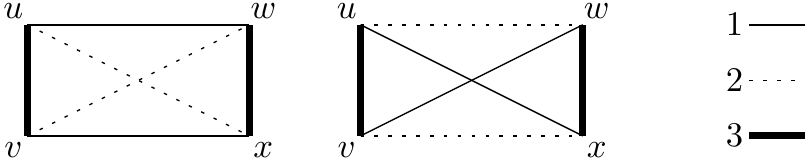} 
\caption{Two possible (isomorphic) antipodal quadruples}
\label{fig:quadruple}
\end{figure}

Suppose that we want to find an EPPA-witness for a single edge of length 3 using Theorem~\ref{thm:hl}. To do it, we in particular need to show that $\mathcal A^3_1$ is a locally finite automorphism-preserving subclass of the class $\mathcal E$ of all $[3]$-edge-labelled graphs, which has  irreducible-structure faithful EPPA by Theorem~\ref{thm:eppanoaxioms}.

However, this does not hold. Take the disjoint union of two edges of length $3$. This clearly has a completion in $\mathcal A^3_1$ (the antipodal quadruple), but it has no automorphism-preserving completion, because one has to pick which edges have length 1 and which edges have length 2.

In order to overcome this issue, we need to expand our structures by some information which will help us decide the parities. At the same time, we have to do it so that there is an expansion $\str A^+$ of $\str A$ such that every partial automorphism of $\str A$ extends to a partial automorphism of $\str A^+$. This allows us to later forget the extra information and get an EPPA-witness for $\str A$.

Let $L^+$ consist of the distance relations $\rel{}{1}$, $\rel{}{2}$ and $\rel{}{3}$, a unary function $M$ and two unary relations, $T$ and $B$ (for \emph{top} and \emph{bottom}), equipped with the permutation group $\Gamma\!_{L^+}$ consisting of the identity and the transposition $(T\; B)$.

Let $\mathcal E$ be the class of all finite $[3]$-edge-labelled graphs where the edges of length $3$ form a perfect matching. Given $\str E\in\mathcal E$, we say that a $\Gamma\!_{L^+}$-structure $\str E^+$ is a \emph{suitable expansion} of $\str E$ if the following hold:
\begin{enumerate}
\item $\str E$ and $\str E^+$ share the same vertices and $\rel{E}{i} = \nbrel{\str E^+}{i}$ for every $1\leq i\leq 3$ (that is, $\str E^+$ and $\str E$ also share the distance relations),
\item $M_{\str E^+}(u) = v$ if and only if $(u,v)\in \nbrel{\str E^+}{3}$ (we need this for the strong amalgamation property), 
\item every vertex of $\str E^+$ is in precisely one of $T_{\str E^+}$ and $B_{\str E^+}$,
\item if $u,v\in E^+$ are connected by an edge of an odd length, then precisely one of $\{u,v\}$ is in $T_{\str E^+}$ and the other is in $B_{\str E^+}$, and
\item if $u,v\in E^+$ are connected by an edge of length 2, then either $\{u,v\}\subseteq T_{\str E^+}$, or $\{u,v\}\subseteq B_{\str E^+}$.
\end{enumerate}
Note that not every $\str E\in \mathcal E$ has a suitable expansion, however, $\str A$ has two of them and both preserve all partial automorphisms of $\str A$ (for this, we need the transposition $(T\; B)$).

Denote by $\mathcal E^+$ the class of all suitable expansions of structures from $\mathcal E$ and similarly define $\mathcal A^{3+}_1$. Theorem~\ref{thm:eppanoaxioms} implies that $\mathcal E^+$ has irreducible-structure faithful EPPA.

In order to prove that $\mathcal A^{3+}_1$ is a locally finite automorphism-preserving subclass of $\mathcal E^+$, it is enough to observe that the conditions of Definition~\ref{defn:locallyfinite} imply that every such $\str B$ which we are asked to complete in fact comes from $\mathcal E^+$, and if we pick $n=6$, we get that it contains no triangles forbidden in $\mathcal A^3_1$. It then suffices to define the missing distances according to the unary relations $T$ and $B$: If $uv$ is not an edge of $\str B$, we put $d(u,v)=2$ if they are in the same unary relation and $d(u,v)=1$ otherwise.

\medskip

In order to prove Proposition~\ref{prop:oddnonbip}, we now generalise the construction above for larger diameters and arbitrary $\str A\in \Aclass$. For the rest of this section, fix $\str A\in \Aclass$. Using Fact~\ref{fact:antipodal}, we can without loss of generality assume that for every vertex $v\in A$ there is a vertex $w\in A$ such that $d_\str{A}(v,w)=\delta$. Enumerate the edges of $\str A$ of length $\delta$ as $e_1,\ldots, e_m$ and let $D = \{1,2,\ldots,m\}$ be their indices, that is, $|D| = \frac{|A|}{2}$ (we will sometimes treat $D$ also as the set $\{e_1,\ldots, e_m\}$ itself using the natural bijection). We furthermore denote $e_i=\{x_i, y_i\}$, where $x_i$ and $y_i$ are vertices of $\str A$.

\subsection{The expanded language}
We will say that a function $\chi\colon D\to \{0,1\}$ is a \emph{valuation function}. For a set $F\subseteq D$, we denote by $\chi^F$ the \emph{flip} of $\chi$, that is, the function $D\to\{0,1\}$ defined as
$$\chi^F(i) = 
\begin{cases}
  1-\chi(i) & \text{ if } i\in F\\
  \chi(i) & \text{ otherwise},
\end{cases}$$
and for a permutation $\psi$ of $D$ we denote by $\chi_\psi$ the function satisfying $\chi_\psi(i) = \chi(\psi^{-1}(i))$. If $\chi$ is a valuation function, $\psi$ is a permutation of $D$ and $F\subseteq D$, then by $\chi_\psi^F$ we will mean $(\chi^F)_\psi$, that is, we first apply the flip and then the permutation.

Let $L$ be the language consisting of binary symmetric irreflexive relations $R^1, \ldots, R^{\delta}$ representing the distances, a unary function $M$, and unary relations $U^\chi_i$ for every $1\leq i\leq m$ and for every valuation function $\chi$. If $\str A$ is an $L$-structure and $v$ is a vertex of $\str A$ such that $v\in U^\chi_i$, we will say that \emph{$v$ has a unary mark $U^\chi_i$}. As in Section~\ref{sec:motivation}, the function $M$ will ensure that the edges of length $\delta$ form a matching, the relations $U^\chi_i$ are generalisations of the relations $T$ and $B$.

Let $F\subseteq D^2$ be such that if $(i,j)\in F$, then also $(j,i)\in F$ ($F$ is \emph{symmetric}). For every $1\leq i\leq m$ we let $F_i\subseteq D$ be the set $\{j\in D:(i,j)\in F\}$. We denote by $\alpha^F$ the permutation of $L$ sending $U_i^\chi\mapsto U_{i}^{\chi^{F_i}}$, which fixes $M$ and $R^1,\ldots,R^{\delta}$ pointwise. In other words, $\alpha^F$ ``flips'' the mutual valuations of pairs from $F$.

For a permutation $\psi$ of $D$, we denote by $\alpha_\psi$ the permutation of $L$ sending $U_i^\chi\mapsto U_{\psi(i)}^{\chi_\psi}$, which fixes $M$ and $R^1,\ldots,R^{\delta}$ pointwise. Now we can define $\GammaL$ as the group generated by
$$\{\alpha^F : F\subseteq D^2\text{ and $F$ is symmetric}\}\cup \{\alpha_\psi:\psi\text{ is a permutation of $D$}\}.$$

\begin{lemma}\label{lem:gammal}
For every member $g\in \GammaL$ there is a permutation $\psi$ of $D$ and a symmetric subset $F\subseteq D^2$ such that $g=\alpha_\psi\alpha^F$.
\end{lemma}
\begin{proof}
Put $$S=\{\alpha^F : F\subseteq D^2\text{ and $F$ is symmetric}\}\cup \{\alpha_\psi:\psi\text{ is a permutation of $D$}\}.$$
We first show three claims:
\begin{claim}\label{cl:gammal:1}
For every $\alpha^F,\alpha^{F'}\in S$ it holds that $\alpha^F\alpha^{F'} = \alpha^{F''}$, where $F''$ is the symmetric difference of $F$ and $F'$ (that is, $(i,j)\in F''$ if and only if it is in exactly one of $F$ and $F'$). Consequently, $\alpha^F\alpha^F = 1$.
\end{claim}
Follows directly from the definitions of $\alpha^F$ and $\chi^F$.
\begin{claim}\label{cl:gammal:2}
For every $\alpha_\psi,\alpha_{\psi'}\in S$ it holds that $\alpha_\psi\alpha_{\psi'} = \alpha_{\psi''}$, where $\psi'' = \psi\psi'$. Consequently, $\alpha_\psi\alpha_{\psi^{-1}}=1$.
\end{claim}
Again follows directly from the definitions of $\alpha_\psi$ and $\chi_\psi$.
\begin{claim}\label{cl:gammal:3}
For every $\alpha_\psi,\alpha^F\in S$ there is $\alpha^{F'}\in S$ such that $\alpha^F\alpha_\psi = \alpha_{\psi}\alpha^{F'}$.
\end{claim}
Put $F'=\psi^{-1}(F)$, that is, $F' = \{(\psi^{-1}(i), \psi^{-1}(j)) : (i,j)\in F\}$. The rest is straightforward verification.

\medskip

We are now ready to prove the statement of this lemma. By definition, every member $g\in \GammaL$ can be written as a word consisting of members of $S$ and their inverses. Using Claims~\ref{cl:gammal:1} and~\ref{cl:gammal:2}, we can replace the inverses by members of $S$, using Claim~\ref{cl:gammal:3} we can ensure that the word can be split into two subwords, first consisting only of $\alpha_\psi$'s and the second consisting of $\alpha^F$'s. From Claims~\ref{cl:gammal:1} and~\ref{cl:gammal:2} it follows that there are $\alpha_\psi,\alpha^F\in S$ such that indeed $g=\alpha_\psi\alpha^F$.
\end{proof}

From now on we will thus denote members of $\GammaL$ by $\alpha_\psi^F$, where 
$$\alpha_\psi^F(U_i^\chi)=\alpha_\psi(\alpha^F(U_i^\chi)) = U_{\psi(i)}^{\chi^{F_i}_\psi},$$
and $\alpha_\psi^F$ is the identity on $\{M,R^1,\ldots,R^{\delta}\}$. In other words, $\alpha_\psi^F$ first ``flips'' the mutual valuations of pairs from $F$ and then permutes the set $D$.

\medskip


For notational convenience, whenever $\str C$ is a $\GammaL$-structure, $U_i^\xi\in L$ and $u\in C$ is a vertex such that $u\in U_i^{\xi}$ and $u$ has no other unary mark, we will denote by $\pi(u) = i$ its \emph{projection} and by $\chi(u) = \xi$ its \emph{valuation}. If $u$ does not have precisely one unary mark, we leave $\pi(u)$ and $\chi(u)$ undefined.

The following (easy) observation says that the unary marks $U_i^\chi$ indeed generalise the construction from Section~\ref{sec:motivation}.
\begin{observation}\label{obs:valuations}
Let $\str C$ be a $\GammaL$-structure such that every vertex of $\str C$ has precisely one unary mark, let $g$ be an automorphism of $\str C$ and let $u,v\in C$ be arbitrary vertices of $\str C$. Then we have
$$\chi(u)(\pi(v)) = \chi(v)(\pi(u))$$
if and only if
$$\chi(g(u))(\pi(g(v))) = \chi(g(v))(\pi(g(u))).$$

This implies that the function $f\colon {C\choose 2}\to \{0,1\}$, defined by $f(uv)=0$ if $\chi(u)(\pi(v)) = \chi(v)(\pi(u))$ and $f(uv)=1$ otherwise, is invariant under $g$ and consequently under all automorphisms of $\str C$.
\end{observation}
\begin{proof}
Assume that $g = (\alpha_\psi^F, g_C)$ and put $F_u = \{j\in D : (\pi(u),j)\in F\}$ and $F_v = \{j\in D : (\pi(v),j)\in F\}$. Since $g$ is an automorphism, we have
$$g(u)\in \alpha_\psi^F(U^{\chi(u)}_{\pi(u)}),$$
hence $\chi(g(u)) = \chi(u)^{F_u}_\psi$ and $\pi(g(u)) = \psi(\pi(u))$ and similarly $\chi(g(v)) = \chi(v)^{F_v}_\psi$ and $\pi(g(v)) = \psi(\pi(v))$.

It follows that
$$\chi(g(u))(\pi(g(v))) = \chi(u)^{F_u}_\psi(\psi(\pi(v))) = \begin{cases}
	1- \chi(u)(\pi(v)) & \text{if }\pi(v)\in F_u\\
	\chi(u)(\pi(v)) &\text{otherwise},
\end{cases}$$
and similarly for $\chi(g(v))(\pi(g(u)))$. Since $F$ is symmetric, we have that $\pi(v)\in F_u$ if and only if $\pi(u)\in F_v$ and thus the claim follows.
\end{proof}

\subsection{The class $\mathcal K$ and completion to it}
Let $\str C\in \Aclass$. We say that a $\GammaL$-structure $\str C^+$ is a \emph{suitable expansion of $\str C$} if the following hold:
\begin{enumerate}
  \item $\str C$ and $\str C^+$ share the same vertex set,
  \item for every $1\leq i\leq \delta$ we have that $\rel{C}{i} = \nbrel{\str C^+}{i}$,
  \item $M_{\str C^+}(u)=v$ if and only if $d_{\str C^+}(u,v)=\delta$,
  \item every vertex of $\str C^+$ has precisely one unary mark,
  \item if $d_{\str C^+}(u,v)=\delta$ and $u\in U_i^\chi$ in $\str C^+$, then $v\in U_i^{1-\chi}$, where $(1-\chi)(j) = 1-\chi(j)$, and
  \item in $\str C^+$ it holds that $\chi(u)(\pi(v)) \neq \chi(v)(\pi(u))$ if and only if $d_{\str C^+}(u,v)$ is odd.
\end{enumerate}
Denote by $\mathcal K$ the class of all suitable expansions of all $\str C\in \Aclass$ where the edges of length $\delta$ form a perfect matching (Fact~\ref{fact:antipodal} says that this is without loss of generality; one can always uniquely and canonically add vertices so that this condition is satisfied). Note that it is possible that there is no suitable expansion of a given $\str C\in\Aclass$.

\begin{prop}\label{prop:nonbippluscompletion}
$\mathcal K$ is a locally finite automorphism-preserving subclass of $\mathcal E$, the class of all finite $\GammaL$-structures.
\end{prop}
\begin{proof}
Let $n$ be a large enough integer (say, at least 4 and at least twice the number of vertices of the largest forbidden cycle in $\mathcal B^\delta_K$) and let $\str A\in \mathcal K$ and $\str B$ be as in Definition~\ref{defn:locallyfinite}. Note that there is an unfortunate notational clash, this $\str A$ is different from the structure $\str A$ which we fixed at the beginning of this section.

The fact that for every $v\in B$ one has that $\cl_{\str B}(v)$ lies in a copy of $\str A$ implies that $M_{\str B}(u)=v$ if and only if $d_{\str B}(u,v)=\delta$ and furthermore the edges of length $\delta$ form a perfect matching in $\str B$ (because this holds in $\str A$).

The fact that every substructure of $\str B$ on at most $n$ vertices has a completion in $\mathcal K$ (which is promised by Definition~\ref{defn:locallyfinite}) implies the following:
\begin{enumerate}
\item Every pair of vertices is in at most one distance relation $R^i$ and these relations are symmetric and irreflexive,
\item every vertex of $\str B$ is in precisely one unary relation, and
\item if $d_{\str B}(u,v)=\delta$ then $v\in U_{\pi(u)}^{1-\chi(u)}$.
\end{enumerate}

We can assume that if $d_{\str B}(u,v)=\delta$ and $w\neq u,v$ is a vertex of $\str B$ such that at least one of $d_{\str B}(u,w)$, $d_{\str B}(v,w)$ is defined, then in fact both distances are defined and furthermore $d_{\str B}(u,w) + d_{\str B}(v,w)=\delta$, because there is a unique way to complete it. It also follows that whenever $u,v$ are vertices such that their distance is defined, then $\chi(u)(\pi(v)) \neq \chi(v)(\pi(u))$ if and only if $d_{\str B}(u,v)$ is odd.

Finally, from the definition of $n$ it also follows that $\str B$ contains no cycles forbidden in $\mathcal B^{\delta}_{K}$ (we needed $n$ to be twice the number of vertices because Definition~\ref{defn:locallyfinite} talks about substructures and these need to be closed for functions). Hence if we define the function $f\colon {B\choose 2}\to \{0,1\}$ as $f(uv)=0$ if $\chi(u)(\pi(v)) = \chi(v)(\pi(u))$ and $f(uv)=1$ otherwise, Fact~\ref{fact:completion} gives us an automorphism-preserving way to add the remaining non-$\delta$ distances, which is exactly what we need for a completion to $\mathcal K$.
\end{proof}
Let us remark that $\mathcal K$ is hereditary: Whenever $\str B$ is a substructure of $\str C\in \Aclass$ such that the edges of length $\delta$ form a perfect matching in both $\str B$ and $\str C$, we have that if $\str C^+$ is a suitable expansion of $\str C$, then the substructure of $\str C^+$ induced on the vertex set $B$ is a suitable expansion of $\str B$. 

\subsection{Constructing the witness}
Recall that at the beginning of this section, we fixed $\str A\in\Aclass$ and enumerated its edges of length $\delta$ as $e_1=\{x_1,y_1\},\ldots,e_m=\{x_m,y_m\}$.

For $1\leq i\leq m$, we define $\chi_i\colon D\to \{0,1\}$ by putting
$$\chi_i(j) = 
\begin{cases}
  1 & \text{ if } i > j \text{ and } d_{\str A}(x_i, x_j) \text{ is odd},\\
  0 & \text{ otherwise}.
\end{cases}$$

We define a suitable expansion $\str A^+\in \mathcal K$ of $\str A$ by putting, for every $1\leq i\leq m$, $M_{\str A^+}(x_i)=y_i$, $M_{\str A^+}(y_i)=x_i$, $x_i\in U_i^{\chi_i}$ and $y_i\in U_i^{1-\chi_i}$. Next we use Theorems~\ref{thm:hl} and~\ref{thm:eppanoaxioms} with Proposition~\ref{prop:nonbippluscompletion} to get $\str B^+\in \mathcal K$ which is an EPPA-witness for $\str A^+$ (so, in particular, $\str A^+\subseteq \str B^+$). Finally, we put $\str B$ to be the \emph{reduct} of $\str B^+$ forgetting all the unary marks and the function $M$. Then indeed, $\str B\in \Aclass$. And since $\str A^+\subseteq \str B^+$, we also have $\str A\subseteq \str B$.

\subsection{Extending partial automorphisms}
We will show that $\str B$ extends all partial automorphisms of $\str A$. Fix a partial automorphism $\varphi$ of $\str A$. Without loss of generality we can assume that whenever $d_\str{A}(u,v)=\delta$ and $u\in \dom(\varphi)$, then also $v\in\dom(\varphi)$ (because there is a unique way of extending $\varphi$ to $v$). Let $\psi\colon D\to D$ be an arbitrary permutation of $D$ extending the action of $\varphi$ on the edges of length $\delta$ of $\str A$.

We now define a set $F\subseteq D^2$ of \emph{flipping pairs}. We put $(i,j)$ and $(j,i)$ in $F$ if $x_i\in\dom(\varphi)$ and $\chi(\varphi(x_i))(\psi(j)) \neq \chi(x_i)(j)$ in $\str A^+$. Note that if both $x_i$ and $x_j$ are in the domain of $\varphi$ then the outcome is the same if we consider $x_j$ instead of $x_i$, because $\varphi$ is an automorphism and therefore preserves the parity of $d_\str{A}(x_i, x_j)$ and thus also the (non)-equality of the corresponding valuations. Note also that if we considered $y_i$ instead of $x_i$, the outcome would still be the same.

What remains is to verify that the pair $(\alpha_\psi^F, \varphi)$ is a partial ($\GammaL$-)automorphism of $\str A^+$. Indeed, assuming that it is the case, we get that it extends to an automorphism $(\theta_L,\theta)$ of $\str B^+$, where $\theta_L = \alpha_\psi^F$ and $\varphi\subseteq \theta$. But this means that $\theta$ is an automorphism of $\str B$ extending $\varphi$ and hence $\str B$ is an EPPA-witness for $\str A$. In the rest of this section we verify that $(\alpha_\psi^F, \varphi)$ is a partial automorphism of $\str A^+$. It amounts to (technical) checking that our construction does what it is supposed to do.

From the fact that $\varphi$ is a partial automorphism of $\str A$ we get that $d_\str{B^+}(u,v)=d_\str{B^+}(\varphi(u),\varphi(v))$ whenever $u,v\in \dom(\varphi)$. This, together with the assumption that whenever $d_\str{A}(u,v)=\delta$ and $u\in \dom(\varphi)$, then also $v\in\dom(\varphi)$, implies that if $v\in \dom(\varphi)$ then $M_\str{B^+}(\varphi(v)) = \varphi(M_\str{B^+}(v))$, or in other words, $\varphi$ respects the function $M$.

It remains to verify that for every $v\in\dom(\varphi)$ and for every $U^\chi_i$ we have $v\in U^\chi_i$ if and only if $\varphi(v)\in \alpha_\psi^F(U^\chi_i)$, or in other words, $\pi(\varphi(v)) = \psi(\pi(v))$ and $\chi(\varphi(v)) = \chi(v)^{F_v}_\psi$, where $F_v = \{j\in D : (\pi(v),j)\in F\}$. Since $\psi$ extends the action of $\varphi$ on the edges of length $\delta$, and since for every $v\in A$ it holds that $\pi(v) = i$ if and only if $v\in e_i$, it follows that for every $v\in \dom(\varphi)$ we have $\pi(\varphi(v)) = \psi(\pi(v))$.

Analogously, from the definition of $F$ we have that $(i,j)$ and $(j,i)$ are in $F$ if and only if $x_i\in\dom(\varphi)$ and $\chi(\varphi(x_i))(\psi(j)) \neq \chi(x_i)(j)$. From the construction it follows that this happens if and only if $y_i\in\dom(\varphi)$ and $\chi(\varphi(y_i))(\psi(j)) \neq \chi(y_i)(j)$. We can summarize these two equivalences as follows: For every $v\in\dom(\varphi)$ and for every $j\in D$ we have $(\pi(v),j)\in F$ if and only if $\chi(\varphi(v))(\psi(j)) \neq \chi(v)(j)$.

By the definition of $\alpha^F$, for every $v\in\dom(\varphi)$ and for every $j\in D$ we have that $\chi(v)^F(j) \neq \chi(v)(j)$ if and only if $(\pi(v),j)\in F$. Consequently, $\chi(v)^F_\psi(\psi(j)) \neq \chi(v)(j)$ if and only if $(\pi(v),j)\in F$, which happens if and only if $\chi(\varphi(v))(\psi(j))\neq \chi(v)(j)$. It follows that $\chi(v)^F_\psi = \chi(\varphi(v))$ which concludes the proof of Proposition~\ref{prop:oddnonbip}.

\subsection{Remarks}
\begin{enumerate}
 \item If we extended the action of $\varphi$ on the edges of length $\delta$ coherently (say, in an order-preserving way), we would get \emph{coherent EPPA} (see~\cite{Siniora}) as in~\cite{eppatwographs}.
 \item The same strategy would also work for proving EPPA for antipodal metric spaces of even diameter, we would only need to pick a subset $O\subset\{0,1,\ldots,\delta\}$ such that $\delta\in O$ and precisely one of $a, \delta-a$ is in $O$ for every $a\in \{0,1,\ldots,\delta\}$ and replace each occurrence of ``odd distance'' by ``distance from $O$'' and ``even distance'' by ``distance $a$ such that $\delta-a\in O$''. (Note that for even $\delta$, we have $\frac{\delta}{2}=\delta-\frac{\delta}{2}$, so $\frac{\delta}{2}$ ``is both odd and even'' in this sense.)
 \item Cherlin also allows to forbid certain sets of $\{1,\delta-1\}$-valued metric spaces (he calls them \emph{Henson constraints}). We chose not to include these classes in order to avoid further technical complications, but using irreducible-structure faithfulness and the fact that the completion from Fact~\ref{fact:completion} does not create distances $1$ and $\delta-1$ gives EPPA also in this case.
\end{enumerate}

\section{The even diameter bipartite case}\label{sec:bip}
The odd diameter bipartite case was done in~\cite{Aranda2017} (because every edge of length $\delta$ has one endpoint in each part of the bipartition and thus there is a unique way of determining parities of the distances, which implies that such classes admit automorphism-preserving completions), so it suffices to deal with the even diameter case. We prove the following proposition.
\begin{prop}\label{prop:bip}
Let $\Aclass$ be a bipartite class of antipodal metric spaces of even diameter. Then for every $\str A\in \Aclass$ there is $\str B\in \Aclass$ which is an EPPA-witness of $\str A$.
\end{prop}

The structure of the proof will be very similar to the odd non-bipartite case. We will also introduce some facts from~\cite{Aranda2017} about completions, add unary functions and unary marks which will help us decide how to fill-in the missing distances while preserving all necessary automorphisms. We have to be a bit more careful in dealing with the bipartiteness (edges of length $\delta$ now lie inside the parts, so we need to make $\psi$ preserve the bipartition, there are also infinitely many forbidden cycles --- the odd perimeter ones), but the general structure is identical.

For the rest of the section, fix $\str A\in \Aclass$. We can without loss of generality assume that every vertex $v\in A$ has some vertex $w\in A$ such that $d_\str{A}(v,w)=\delta$. Consider the set $\{e_1,\ldots, e_m\}$ of edges of $\str A$ of length $\delta$ and let $D = \{1,2,\ldots,m\}$ be their indices, that is, $|D| = \frac{|A|}{2}$. We denote $e_i=\{x_i, y_i\}$, where $x_i$ and $y_i$ are vertices of $\str A$. Since $\str A$ is bipartite, we have that the relation ``vertices $u$ and $v$ are at an even distance'' is an equivalence relation on $\str A$ which has two equivalence classes. Because $\delta$ is even, we can assume that $D = D_1\cup D_2$, where $D_1$ consists of the indices of edges with both endpoints in one part and $D_2$ consists of the indices of edges with both endpoints in the other part.

We also assume without loss of generality that $|D_1| = |D_2|$ (otherwise we can add more vertices to $\str A$, and if this larger structure has an EPPA-witness $\str B$, then it is also an EPPA-witness of the original $\str A$).

We will need the following analogue of Fact~\ref{fact:completion}. 
\begin{fact}\label{fact:bipcompletion}
Let $\Aclass$ be a bipartite class of antipodal metric spaces. Let $\str A$ be a $[\delta]$-edge-labelled graph such that the edges of length $\delta$ of $\str A$ form a perfect matching and furthermore for every $u,v,w\in A$ such that $d_\str{A}(u,v)=\delta$ and $w\neq u,v$, either $w$ is not connected by an edge to either of $u,v$, or $d_\str{A}(u,w)+d_\str{A}(v,w)=\delta$. Suppose furthermore that $\str A$ contains no odd-perimeter cycles and none of the finitely many even-perimeter cycles forbidden in $\mathcal B^{\delta}_{K}$.

Let $O\subset \{0,1,\ldots,\delta\}$ be a set such that $\delta\in O$ and exactly one of $a,\delta-a$ is in $O$ for every $a\in\{0,1,\ldots,\delta\}$ and denote by $\delta-O$ the set $\{\delta - a : a\in O\}$.

Let $f\colon {A\choose 2}\to \{0,1\}$ be a mapping satisfying the following.
\begin{enumerate}
  \item Whenever $uv$ is an edge of $\str A$, then $f(uv) = 1$ implies that $d_\str{A}(u,v)\in O$ and $f(uv)=0$ implies that $d_\str{A}(u,v)\in \delta-O$.\footnote{This seemingly sloppy statement is necessary in order to deal with $\frac{\delta}{2}$ being in both $O$ and $\delta-O$ for even $\delta$.}
  \item Let $u_1v_1$ and $u_2v_2$ be two different edges of length $\delta$ of $\str A$. Then $f(u_1u_2)=f(v_1v_2)$, $f(u_1v_2)=f(u_2v_1)$ and $f(u_1u_2)\neq f(u_1v_2)$.
\end{enumerate}

Then there is $\overbar{\str A}\in \Aclass$ such that the following holds.
\begin{enumerate}
  \item $\overbar{\str A}$ is a completion of $\str A$ with the same vertex set,
  \item for every edge $uv$ of $\overbar{\str A}$ it holds that $f(uv) = 1$ implies that $d_{\overbar{\str{A}}}(u,v)\in O$ and $f(uv)=0$ implies that $d_{\overbar{\str{A}}}(u,v)\in \delta-O$, and
  \item Every automorphism of $\str A$ which preserves values of $f$ is also an automorphism of $\overbar{\str A}$.
\end{enumerate}
\end{fact}

\subsection{The expanded language}
As in the odd non-bipartite case, we will call a function $\chi\colon D\to \{0,1\}$ a \emph{valuation function}, adopt the same notions of \emph{flips} $\chi^F$ and permutations $\chi_\psi$. We also let $L$ be the same language as in Section~\ref{sec:nonbip}, adding a unary function $M$ and unary relations $U^\chi_i$.

In contrast to Section~\ref{sec:nonbip}, we put $\GammaL$ to be the group generated by
$$S=\{\alpha^F : F\subseteq D^2\text{ and $F$ is symmetric}\}\cup \{\alpha_\psi:\psi\text{ is a partition-preserving permutation of $D$}\},$$
where $\psi$ is \emph{partition-preserving} if either $\psi(D_1)=D_1$ and $\psi(D_2)=D_2$, or $\psi(D_1)=D_2$ and $\psi(D_2)=D_1$. Analogously to Lemma~\ref{lem:gammal} it follows that every element of $\GammaL$ can be written as the product $\alpha_\psi\alpha^F$, where $\alpha_\psi,\alpha^F\in S$. We will denote $\alpha_\psi^F = \alpha_\psi\alpha^F$.

Again, for a vertex $u$ in a $\GammaL$-structure which has precisely one unary mark $U_i^{\xi}$, we define $\pi(u)=i$ and $\chi(u)=\xi$ and we have the same observation with the same proof as before.
\begin{observation}\label{obs:bipvaluations}
Let $\str C$ be a $\GammaL$-structure such that every vertex of $\str C$ has precisely one unary mark, let $g$ be an automorphism of $\str C$ and let $u,v\in C$ be arbitrary vertices of $\str C$. Then we have
$$\chi(u)(\pi(v)) = \chi(v)(\pi(u))$$
if and only if
$$\chi(g(u))(\pi(g(v))) = \chi(g(v))(\pi(g(u))).$$

This implies that the function $f\colon {C\choose 2}\to \{0,1\}$, defined by $f(uv)=0$ if $\chi(u)(\pi(v)) = \chi(v)(\pi(u))$ and $f(uv)=1$ otherwise, is invariant under $g$ and consequently under all automorphisms of $\str C$.
\end{observation}
Note that the edge-labelled graph formed by the distance relations in $\str C$ may contain odd cycles.

\subsection{The class $\mathcal K$ and completion to it}
Now we also have to ensure that structures from $\mathcal K$ are bipartite. Let $\str C\in \Aclass$. Since $\str C$ is bipartite, we can denote by $Q_1,Q_2$ its parts (that is, $Q_1\cup Q_2 = C$ and each of $Q_1,Q_2$ is an equivalence class of the relation ``vertices $u$ and $v$ are at an even distance from each other''). We say that a $\GammaL$-structure $\str C^+$ is a \emph{suitable expansion of $\str C$} if the following hold:
\begin{enumerate}
  \item $\str C$ and $\str C^+$ share the same vertex set,
  \item for every $1\leq i\leq \delta$ we have that $\rel{C}{i} = \nbrel{\str C^+}{i}$,
  \item $M_{\str C^+}(u)=v$ if and only if $d_{\str C^+}(u,v)=\delta$,
  \item every vertex of $\str C^+$ has precisely one unary mark,
  \item if $d_{\str C^+}(u,v)=\delta$ and $u\in U_i^\chi$ in $\str C^+$, then $v\in U_i^{1-\chi}$,
  \item in $\str C^+$ it holds that if $\chi(u)(\pi(v)) \neq \chi(v)(\pi(u))$ then $d_{\str C^+}(u,v) \in O$ and if $\chi(u)(\pi(v)) = \chi(v)(\pi(u))$ then $d_{\str C^+}(u,v) \in \delta-O$, and
  \item let $P_1=\{v\in C : \pi(v)\in D_1\}$ and $P_2=\{v\in C : \pi(v)\in D_2\}$ (where $\pi$ is taken with respect to $\str C^+$). Then either $P_1=Q_1$ and $P_2=Q_2$, or $P_1=Q_2$ and $P_2=Q_1$.
\end{enumerate}
Denote by $\mathcal K$ the class of all suitable expansions of all $\str C\in \Aclass$ where the edges of length $\delta$ form a perfect matching.

\begin{prop}\label{prop:bippluscompletion}
$\mathcal K$ is a locally finite automorphism-preserving subclass of $\mathcal E$, the class of all finite $\GammaL$-structures.
\end{prop}
\begin{proof}
Let $n$ be a large enough integer (say, at least 4 and at least twice the number of vertices of the largest even-perimeter forbidden cycle in $\mathcal B^{\delta}_{K}$) and let $\str A\in \mathcal K$ and $\str B$ be as in Definition~\ref{defn:locallyfinite}. (Again, this is not the $\str A$ which we fixed at the beginning of this section.)

As for the odd non-bipartite case we get the following:
\begin{enumerate}
\item Every vertex of $\str B$ is in precisely one unary relation,
\item every pair of vertices is in at most one distance relation $R^i$ (and these relations are symmetric),
\item $M_\str{B}(u)=v \iff d_{\str B}(u,v)=\delta$,
\item the edges of length $\delta$ form a perfect matching in $\str B$,
\item if $d_{\str B}(u,v)=\delta$ then $v\in U_{\pi(u)}^{1-\chi(u)}$,
\item without loss of generality, we can assume that if $d_{\str B}(u,v)=\delta$ and $w\neq u,v$ is a vertex of $\str B$ such that at least one of $d_{\str B}(u,w)$, $d_{\str B}(v,w)$ is defined, then both distances are defined and furthermore $d_{\str B}(u,w) + d_{\str B}(v,w)=\delta$.
\item let $u,v$ be vertices such that their distance is defined. Then $\chi(u)(\pi(v)) \neq \chi(v)(\pi(u))$ implies $d_{\str A^+}(u,v)\in O$ and $\chi(u)(\pi(v)) = \chi(v)(\pi(u))$ implies $d_{\str A^+}(u,v)\in \delta-O$. 
\end{enumerate}

Furthermore, from the last condition for a suitable expansion we also get that two vertices $u,v$ of $\str B$ are at an even distance, if and only if there is $i\in \{1,2\}$ such that $\pi(u),\pi(v)\in D_i$. Note that this implies that $\str B$ contains no cycles of odd perimeter (each cycle has to contain an even number of odd edges).

Finally, from the definition of $n$ it also follows that $\str B$ contains no even-perimeter cycles forbidden in $\mathcal B^{\delta}_{K}$. Hence if we define the function $f\colon {B\choose 2}\to \{0,1\}$ as $f(uv)=0$ if $\chi(u)(\pi(v)) = \chi(v)(\pi(u))$ and $f(uv)=1$ otherwise, Fact~\ref{fact:bipcompletion} gives us an automorphism-preserving way to add the remaining distances, which is exactly what we need for a completion to $\mathcal K$.
\end{proof}
Let us again remark that $\mathcal K$ is hereditary.

\subsection{Constructing the witness}
This is completely the same as for the odd diameter non-bipartite case. We define a $\GammaL$-structure $\str A^+$ which is a suitable expansion of $\str A$, and use Theorems~\ref{thm:hl} and~\ref{thm:eppanoaxioms} with Proposition~\ref{prop:bippluscompletion} to get $\str B^+\in \mathcal K$ which is an EPPA-witness for $\str A^+$. Finally, we put $\str B$ to be the reduct of $\str B^+$ forgetting all unary marks and all functions $M$.

\subsection{Extending partial automorphisms}
Again, this is completely the same as before with the exception that the permutation $\psi$ of $D$ has to preserve the bipartition $D=D_1\cup D_2$ (it can exchange $D_1$ and $D_2$). Every partial automorphism $\varphi$ of $\str A$ respects the bipartition, and since we assumed that $|D_1|=|D_2|$, it is always possible to extend $\varphi$ to a full permutation $\psi$ as needed.

Let us remark that if one is a bit more careful, the same strategy again gives coherent EPPA.

\section{Conclusion}
We are now ready to prove Theorem~\ref{thm:main}.
\begin{proof}[Proof of Theorem~\ref{thm:main}]
In~\cite{Aranda2017}, EPPA is proved for non-bipartite classes of even diameter and bipartite classes of odd diameter. Proposition~\ref{prop:oddnonbip} proves EPPA for non-bipartite classes of odd diameter and Proposition~\ref{prop:bip} proves EPPA for bipartite classes of even diameter, hence Theorem~\ref{thm:main} is proved.
\end{proof}

We think of this paper as the first example of a more general method for bypassing the lack of an automorphism-preserving completion, namely using the method of valuation functions to add more information to the structures (and thus restrict automorphisms) while preserving all partial automorphisms of one given structure $\str A$, and then plugging this expanded class into the existing machinery. A similar trick can be done also for structures with higher arities, using higher-arity valuation functions (cf.~\cite{Hubicka2018EPPA}). However, there are still classes where this method does not work, for example the class of tournaments which poses a long-standing important problem in this area.

\section{Acknowledgements}
I would like to thank Jan Hubi\v cka, Jaroslav Ne\v set\v ril and Gregory Cherlin for valuable advice and comments which significantly improved this paper. I would also like to thank all three anonymous referees for their incredibly helpful comments. This research was supported by project 18-13685Y of the Czech Science Foundation (GA\v CR) and by Charles University project GA UK No 378119.

\bibliography{ramsey.bib}

\begin{thebibliography}{10}

\bibitem{Amato2016}
Daniela Amato, Gregory Cherlin, and Dugald Macpherson.
\newblock Metrically homogeneous graphs of diameter three.
\newblock preprint, 2016.

\bibitem{Aranda2017a}
Andres Aranda, David Bradley-Williams, Eng~Keat Hng, Jan Hubi{\v c}ka,
  Miltiadis Karamanlis, Michael Kompatscher, Mat{\v e}j Kone{\v c}n{\'y}, and
  Micheal Pawliuk.
\newblock Completing graphs to metric spaces.
\newblock {\em Electronic Notes in Discrete Mathematics}, 61:53 -- 60, 2017.
\newblock The European Conference on Combinatorics, Graph Theory and
  Applications (EUROCOMB'17).

\bibitem{Aranda2017c}
Andres Aranda, David Bradley-Williams, Eng~Keat Hng, Jan Hubi{\v c}ka,
  Miltiadis Karamanlis, Michael Kompatscher, Mat{\v e}j Kone{\v c}n{\'y}, and
  Micheal Pawliuk.
\newblock Completing graphs to metric spaces.
\newblock Submitted, arXiv:1706.00295, 2017.

\bibitem{Aranda2017}
Andres Aranda, David Bradley-Williams, Jan Hubi{\v c}ka, Miltiadis Karamanlis,
  Michael Kompatscher, Mat{\v e}j Kone{\v c}n{\'y}, and Micheal Pawliuk.
\newblock Ramsey expansions of metrically homogeneous graphs.
\newblock Submitted, arXiv:1707.02612, 2017.

\bibitem{Biggs1971}
Norman~L. Biggs and Donald~H. Smith.
\newblock On trivalent graphs.
\newblock {\em Bulletin of the London Mathematical Society}, 3(2):155--158,
  1971.

\bibitem{Cherlin2011b}
Gregory Cherlin.
\newblock Two problems on homogeneous structures, revisited.
\newblock {\em Model theoretic methods in finite combinatorics}, 558:319--415,
  2011.

\bibitem{Cherlin2013}
Gregory Cherlin.
\newblock Homogeneous ordered graphs and metrically homogeneous graphs.
\newblock In preparation, December 2017.

\bibitem{Conant2015}
Gabriel Conant.
\newblock Extending partial isometries of generalized metric spaces.
\newblock {\em Fundamenta Mathematicae}, 244:1--16, 2019.

\bibitem{eppatwographs}
David~M. Evans, Jan Hubi{\v{c}}ka, Mat{\v {e}}j Kone{\v {c}}n{\'{y}}, and
  Jaroslav Ne\v{s}et\v{r}il.
\newblock E{P}{P}{A} for two-graphs and antipodal metric spaces.
\newblock Submitted, arXiv:1812.11157, 2018.

\bibitem{Fraisse1953}
Roland Fra{\"\i}ss{\'e}.
\newblock Sur certaines relations qui g\'en\'eralisent l'ordre des nombres
  rationnels.
\newblock {\em Comptes Rendus de l'Academie des Sciences}, 237:540--542, 1953.

\bibitem{Gardiner1976}
Anthony Gardiner.
\newblock Homogeneous graphs.
\newblock {\em Journal of Combinatorial Theory, Series B}, 20(1):94--102, 1976.

\bibitem{Godsil2001}
Chris {Godsil} and Gordon {Royle}.
\newblock {\em Algebraic Graph Theory}, volume 207 of {\em Graduate Texts in
  Mathematics.}
\newblock volume 207 of Graduate Texts in Mathematics. Springer, 2001.

\bibitem{Herwig1995}
Bernhard Herwig.
\newblock Extending partial isomorphisms on finite structures.
\newblock {\em Combinatorica}, 15(3):365--371, 1995.

\bibitem{herwig1998}
Bernhard Herwig.
\newblock Extending partial isomorphisms for the small index property of many
  $\omega$-categorical structures.
\newblock {\em Israel Journal of Mathematics}, 107(1):93--123, 1998.

\bibitem{herwig2000}
Bernhard Herwig and Daniel Lascar.
\newblock Extending partial automorphisms and the profinite topology on free
  groups.
\newblock {\em Transactions of the American Mathematical Society},
  352(5):1985--2021, 2000.

\bibitem{hodges1993b}
Wilfrid Hodges, Ian Hodkinson, Daniel Lascar, and Saharon Shelah.
\newblock The small index property for $\omega$-stable $\omega$-categorical
  structures and for the random graph.
\newblock {\em Journal of the London Mathematical Society}, 2(2):204--218,
  1993.

\bibitem{hodkinson2003}
Ian Hodkinson and Martin Otto.
\newblock Finite conformal hypergraph covers and {G}aifman cliques in finite
  structures.
\newblock {\em Bulletin of Symbolic Logic}, 9(03):387--405, 2003.

\bibitem{hrushovski1992}
Ehud Hrushovski.
\newblock Extending partial isomorphisms of graphs.
\newblock {\em Combinatorica}, 12(4):411--416, 1992.

\bibitem{Hubickacycles2018}
Jan Hubi{\v c}ka, Michael Kompatscher, and Mat{\v e}j Kone{\v c}n{\'y}.
\newblock Forbidden cycles in metrically homogeneous graphs.
\newblock Submitted, arXiv:1808.05177, 2018.

\bibitem{Hubicka2018metricEPPA}
Jan Hubi{\v{c}}ka, Mat{\v{e}}j Kone{\v{c}}n{\'y}, and Jaroslav
  Ne{\v{s}}et{\v{r}}il.
\newblock A combinatorial proof of the extension property for partial
  isometries.
\newblock {\em Commentationes Mathematicae Universitatis Carolinae},
  60(1):39--47, 2019.

\bibitem{Hubicka2017sauer}
Jan Hubi{\v{c}}ka, Mat{\v {e}}j Kone{\v {c}}n{\'{y}}, and Jaroslav
  Ne\v{s}et\v{r}il.
\newblock Semigroup-valued metric spaces: {R}amsey expansions and {E}{P}{P}{A}.
\newblock In preparation, 2018.

\bibitem{Hubicka2018EPPA}
Jan Hubi{\v{c}}ka, Mat\v{e}j Kone\v{c}n\'{y}, and Jaroslav Ne\v{s}et\v{r}il.
\newblock All those {E}{P}{P}{A} classes (strengthenings of the
  {H}erwig-{L}ascar theorem).
\newblock arXiv:1902.03855, 2019.

\bibitem{Hubicka2016}
Jan Hubi{\v{c}}ka and Jaroslav Ne\v{s}et\v{r}il.
\newblock All those {R}amsey classes ({R}amsey classes with closures and
  forbidden homomorphisms).
\newblock Submitted, arXiv:1606.07979, 58 pages, 2016.

\bibitem{Konecny2018bc}
Mat{\v e}j Kone{\v c}n{\'y}.
\newblock Combinatorial properties of metrically homogeneous graphs.
\newblock Bachelor's thesis, Charles University, 2018.
\newblock arXiv:1805.07425.

\bibitem{Konecny2018b}
Mat{\v e}j Kone{\v c}n{\'y}.
\newblock Semigroup-valued metric spaces.
\newblock Master's thesis, Charles University, 2019.
\newblock arXiv:1810.08963.

\bibitem{Lachlan1980}
Alistair~H. Lachlan and Robert~E. Woodrow.
\newblock Countable ultrahomogeneous undirected graphs.
\newblock {\em Transactions of the American Mathematical Society}, pages
  51--94, 1980.

\bibitem{mackey1966}
George~W. Mackey.
\newblock Ergodic theory and virtual groups.
\newblock {\em Mathematische Annalen}, 166(3):187--207, 1966.

\bibitem{Macpherson2011}
Dugald Macpherson.
\newblock A survey of homogeneous structures.
\newblock {\em Discrete Mathematics}, 311(15):1599 -- 1634, 2011.
\newblock Infinite Graphs: Introductions, Connections, Surveys.

\bibitem{Nevsetvril2007}
Jaroslav Ne{\v{s}}et{\v{r}}il.
\newblock Metric spaces are {R}amsey.
\newblock {\em European Journal of Combinatorics}, 28(1):457--468, 2007.

\bibitem{otto2017}
Martin Otto.
\newblock Amalgamation and symmetry: From local to global consistency in the
  finite.
\newblock {\em arXiv:1709.00031}, 2017.

\bibitem{sabok2017automatic}
Marcin Sabok.
\newblock Automatic continuity for isometry groups.
\newblock {\em Journal of the Institute of Mathematics of Jussieu}, pages
  1--30, 2017.

\bibitem{Siniora}
Daoud Siniora and S{\l}awomir Solecki.
\newblock Coherent extension of partial automorphisms, free amalgamation, and
  automorphism groups.
\newblock {\em The Journal of Symbolic Logic}, 2019.

\bibitem{solecki2005}
S{\l}awomir Solecki.
\newblock Extending partial isometries.
\newblock {\em Israel Journal of Mathematics}, 150(1):315--331, 2005.

\bibitem{vershik2008}
Anatoly~M. Vershik.
\newblock Globalization of the partial isometries of metric spaces and local
  approximation of the group of isometries of {U}rysohn space.
\newblock {\em Topology and its Applications}, 155(14):1618--1626, 2008.

\end{thebibliography}
\end{document}